\documentclass[10pt, leqno, a4paper]{amsart}

\usepackage[latin1]{inputenc}
\usepackage{amsfonts}
\usepackage{amsmath}
\usepackage{amssymb}
\usepackage{amsthm}
\usepackage[T1]{fontenc}
\usepackage[dvips]{graphicx}
\usepackage{enumerate}
\usepackage{verbatim}
\usepackage{lmodern}

\usepackage{hyperref}
\usepackage[matrix, arrow]{xy}

\input xy
\xyoption{all}
\SelectTips{cm}{10}

\DeclareMathOperator{\Proj}{Proj}

\DeclareMathOperator{\Syz}{Syz}

\DeclareMathOperator{\height}{ht}

\DeclareMathOperator{\Tr}{Tr}

\DeclareMathOperator{\chara}{char}
\DeclareMathOperator{\Spec}{Spec}

\DeclareMathOperator{\depth}{depth}

\newcommand{\eps}{\varepsilon}

\begin{document}

\swapnumbers
\theoremstyle{definition}
\newtheorem{Le}{Lemma}[section]
\newtheorem{Def}[Le]{Definition}
\newtheorem*{DefB}{Definition}
\newtheorem{Bem}[Le]{Remark}
\newtheorem{Ko}[Le]{Corollary}
\newtheorem{Theo}[Le]{Theorem}
\newtheorem*{TheoB}{Theorem}
\newtheorem{Bsp}[Le]{Example}
\newtheorem{Be}[Le]{Observation}
\newtheorem{Prop}[Le]{Proposition}
\newtheorem{Sit}[Le]{Situation}
\newtheorem{Que}[Le]{Question}
\newtheorem*{Con}{Conjecture}
\newtheorem{Dis}[Le]{Discussion}
\newtheorem{Prob}[Le]{Problem}
\newtheorem{Konv}[Le]{Convention}
\title{Dagger closure in regular rings containing a field}
\author{Holger Brenner}
\author{Axel St\"abler}
\address{Holger Brenner, Universit\"at Osnabr\"uck, Fachbereich 6: Mathematik/Informatik,
Albrechtstr. 28a,
49069 Osnabr\"uck, Germany}
\email{hbrenner@uni-osnabrueck.de}

\address{Axel St\"abler, Universit\"at Osnabr\"uck, Fachbereich 6: Mathematik/Informatik,
Albrechtstr. 28a,
49069 Osnabr\"uck, Germany}
\curraddr{Johannes Gutenberg-Universit\"at Mainz\\ Fachbereich 08\\
Staudingerweg 9\\
55099 Mainz\\
Germany}
\email{axel.staebler@uni-osnabrueck.de}


\subjclass[2010]{Primary 13A35, 13A18; Secondary 13D45}
\begin{abstract}
We prove that dagger closure is trivial in regular domains containing a field and that graded dagger closure is trivial in polynomial rings over a field. We also prove that Heitmann's full rank one closure coincides with tight closure in positive characteristic under some mild finiteness conditions. Furthermore, we prove that dagger closure is always contained in solid closure and that the forcing algebra for an element contained in dagger closure is parasolid.
\end{abstract}

\maketitle

\section*{Introduction}
Dagger closure is an attempt to provide a tight closure like closure operation that is valid in any characteristic without reducing to positive characteristic. More precisely, an element belongs to the dagger closure of an ideal in a domain $R$ if it is multiplied into the extended ideal in the absolute integral closure (i.\,e.\ the integral closure of $R$ in an algebraic closure of its field of fractions) by ``arbitrarily small'' elements. In order to make sense of the notion ``arbitrarily small'' valuations are used.

Dagger closure was first introduced by Hochster and Huneke in \cite{hochsterhunekedagger} for a complete local domain of positive characteristic. There they also proved that dagger closure coincides with tight closure in this setting [ibid.,Theorem 3.1]. Also Heitmann's \emph{full rank one closure} which he used to prove the direct summand conjecture in mixed characteristic in dimension $3$ (cf.\ \cite{heitmanndirectsummand}) is a variant of dagger closure tailored to mixed characteristics.

One important feature of such a tight closure like operation should be that it is trivial on ideals in regular rings, since this property opens relations to singularity theory. It is also a crucial feature with respect to the direct summand conjecture in mixed characteristic. 
In this paper we prove this property for dagger closure for regular rings containing a field (see Theorem \ref{daggerregularclose}). The argument we provide uses the existence of big Cohen Macaulay algebras.  

In \cite{brennerstaeblerdaggersolid} we prove that a graded variant of dagger closure agrees with solid closure in graded dimension $2$. And in \cite{staeblerdaggerabelian} the second author proves an inclusion result for certain section rings of abelian varieties. This implies in particular that dagger closure is non-trivial in all dimensions.

This article is based on parts of the Ph.D.\ thesis of the second author (\cite{staeblerthesis}). Related ideas concerning dagger closure and ``almost zero'' have also been studied in \cite{robertssinghannihilators} and \cite{asgbhattaremakrsonbcm}. 

\section{Preliminaries}
In this section we recall the original definition of dagger closure by Hochster and Huneke and the graded version of dagger closure as defined in \cite{brennerstaeblerdaggersolid}. Then we introduce a definition of dagger closure for arbitrary domains and relate it and the graded version to the original definition of Hochster and Huneke. We will also prove that they both coincide with tight closure in positive characteristic. 
Our definition is in principle applicable if the ring does not contain a field but we will not prove any results in this case.

We recall that for a noetherian domain $R$ the absolute integral closure $R^+$ of $R$ is the integral closure of $R$ in an algebraic closure of its field of fractions $Q(R)$. By a result of Hochster and Huneke (\cite[Lemma 4.1]{hochsterhunekeinfinitebig}) there is for an $\mathbb{N}$-graded domain $R$ a maximal $\mathbb{Q}_{\geq 0}$-graded subring of $R^+$ which extends the grading of $R$ -- we will denote this ring by $R^{+ \text{GR}}$.

We can now state the definition of Hochster and Huneke.

\begin{Def}
\label{daggerHH}
Let $(R, \mathfrak{m})$ be a complete local domain. Fix a valuation $\nu$ on $R$ with values in $\mathbb{Z}$ such that $\nu(R) \geq 0$ and $\nu(\mathfrak{m}) >0$. Extend $\nu$ to $R^+$ so that it takes values in $\mathbb{Q}$. Then an element $f \in R$ belongs to the \emph{dagger  closure $I^\dagger$} of an ideal $I \subseteq R$ if for all positive $\eps$ there exists an element $a \in R^+$ with $\nu(a) < \eps$ such that $af \in IR^+$.
\end{Def}

\begin{Prop}
\label{PValuationbyGrading}
Let $R$ be a $\mathbb{Q}$-graded domain. The map $\nu: R \setminus \{ 0 \} \to \mathbb{Q}$ sending $f \in R \setminus \{0\}$ to $\deg f_i$, where $f_i$ is the minimal homogeneous component of $f$ induces a valuation on $R$ with values in $\mathbb{Q}$. This valuation will be referred to as the \emph{valuation induced by the grading}.
\end{Prop}
\begin{proof}
Standard.
\end{proof}

\begin{Def}
\label{Defgradedagger}
Let $R$ denote an $\mathbb{N}$-graded domain and let $I$ be an ideal of $R$. Let $\nu$ be the valuation on $R^{+ \text{GR}}$ induced by the grading on $R$. Then an element $f$ belongs to the \emph{graded dagger closure} $I^{\dagger \text{GR}}$ of an ideal $I$ if for all positive $\varepsilon$ there exists an element $a \in R^{+ \text{GR}}$ with $\nu(a) < \varepsilon$ such that $af \in IR^{+ \text{GR}}$.
If $R$ is not a domain we say that $f \in I^{\dagger \text{GR}}$ if $f \in (IR/P)^{\dagger \text{GR}}$ for all minimal primes $P$ of $R$.\footnote{Note that this is well-defined since the minimal primes are homogeneous (see \cite[Lemma 1.5.6 (a)]{brunsherzog}).}
\end{Def}

\begin{Def}
\label{dagger}
Let $R$ be a domain and $I$ an ideal. Then an element $f$ belongs to the \emph{dagger closure} $I^\dagger$ of $I$ if for every valuation $\nu$ of rank at most one on $R^+$ and every positive $\eps$ there exists $a \in R^+$ with $\nu(a) < \eps$ and $af \in IR^+$.
\end{Def} 

We recall that a valuation of rank one is a non-trivial valuation whose value group is contained in $\mathbb{R}$.
A valuation of rank zero is precisely the trivial valuation (see \cite[vol.\,2, VI \S 10 Theorem 15]{zariskisamuel}) and we include this only to force that $0^\dagger = 0$ if $R$ is a field of positive characteristic. Indeed, if for instance $R = \mathbb{F}_p$ then there is no non-trivial valuation on $R$. See also \cite[Remark 1.4]{heitmannplusextended}.
Henceforth, the term ``dagger closure'' refers to Definition \ref{dagger} unless explicitly stated otherwise.

\begin{Bem}
\label{negvaldagger}
In the definition of dagger closure it is enough to only consider valuations $\nu$ such that $R^+ \subseteq R_\nu$, where $R_\nu$ is the valuation ring associated to $\nu$. Indeed, $R^+ \subseteq R_\nu$ means that all elements of $R^+$ have non-negative valuation. If there is an element $a \in R^+$ such that $\nu(a) < 0$ and $0 \neq I \subseteq R$ is an ideal then $I^\dagger = R$ with respect to $\nu$. To see this let $0 \neq f \in I$. Then $a^n f \cdot 1 \in I$ and for $n$ sufficiently large $\nu(a^n f) < 0$.

If $I = 0$ then $I^\dagger = 0$ in any case.
\end{Bem}

\begin{Le}
\label{Lvaluation}
Let $(R, \mathfrak{m}, k)$ be a complete local noetherian domain containing a field. Then there exists a $\mathbb{Q}$-valued valuation $\nu$ on $R$ which is non-negative on $R$ and positive on $\mathfrak{m}$.
\end{Le}
\begin{proof}
As $R$ is a noetherian integral domain it is dominated by a discrete valuation ring -- cf. \cite[Theorem 6.3.3]{swansonhunekeintegralclosure}, the case $R$ a field is trivial.
\end{proof}

\begin{Le}
\label{daggerextension}
Let $R \subseteq S$ be an extension of integral domains and $I$ an ideal. Then $I^\dagger \subseteq (IS)^\dagger \cap R$. 
\end{Le}
\begin{proof}
Follows similarly to \cite[Lemma 1.6]{heitmannplusextended}.
\end{proof}

\begin{Prop}
\label{daggertightexcellent}
Let $R$ be a domain of characteristic $p > 0$ essentially of finite type over an excellent local ring. Then dagger closure as in Definition \ref{dagger} coincides with tight closure.
\end{Prop}
\begin{proof}
Dagger closure always contains tight closure (see the argument in the proof of Theorem \cite[Theorem 3.1]{hochsterhunekedagger} on p.\ 235) so we only have to show the other inclusion. Passing to the normalisation we may assume that $R$ is normal. 
By virtue of \cite[Theorem 1.4.11]{hochsterhuneketightzero} $R$ has a completely stable test element. Hence, by \cite[Theorem 1.4.7 (g)]{hochsterhuneketightzero} an element $f$ is contained in the tight closure of an ideal $I$ if and only if $f$ is contained in the tight closure of the extended ideal $I \widehat{R_\mathfrak{m}}$ for the completion of $R$ at every maximal ideal $\mathfrak{m}$. Furthermore, Lemma \ref{daggerextension} yields that $I^\dagger \subseteq (I\widehat{R_\mathfrak{m}})^\dagger$. Therefore, we may assume that $R$ is a complete local domain (since $R$ is normal and excellent, its completions at maximal ideals are again domains by \cite[7.8.3 (vii)]{EGAIV}).

Let now $I \subseteq R$ be an ideal, $f \in R$ and assume that $f \in I^\dagger$. In particular, this implies that $f$ is multiplied into $IR^+$ by elements of small order with respect to an extension of a valuation as in Lemma \ref{Lvaluation} above. But by \cite[Theorem 3.1]{hochsterhunekedagger} this yields that $f \in I^\ast$.
\end{proof}

The following result compares our definition of dagger closure with that of Hochster and Huneke.

\begin{Ko}
\label{KoNervensaege}
Let $(R,\mathfrak{m})$ be a complete local noetherian domain. Then dagger closure is contained in the dagger closure in the sense of Definition \ref{daggerHH}. If $R$ is of positive characteristic or if the ideal is primary to $\mathfrak{m}$ then the two closures coincide.
\end{Ko}
\begin{proof}
If an element is contained in dagger closure then a fortiori in the dagger closure of Hochster and Huneke. In positive characteristic the result is immediate from Proposition \ref{daggertightexcellent} and \cite[Theorem 3.1]{hochsterhunekedagger}.

Assume now that the ideal $I \subseteq R$ is primary. We have already seen in Remark \ref{negvaldagger} that we only need to look at valuations that are non-negative on $R$. If they are in addition positive on $\mathfrak{m}$ then they are equivalent by a theorem of Izumi (\cite{izumimeasure}). So the only remaining case is that $\nu(\mathfrak{m}) \geq 0$ and there is an $a \in \mathfrak{m}$ such that $\nu(a) = 0$. Since $I$ is primary $a^n \in I$ for $n \gg 0$ so with respect to this fixed valuation the dagger closure of $I$ is $R$.
\end{proof}

Of course, it is interesting to ask how Definition \ref{dagger} compares to the graded version of dagger closure (Definition \ref{Defgradedagger}). We do have a result in positive characteristic. First of all, we need the following

\begin{Le}
\label{gradedvaluationextension}
Let $R$ be an $\mathbb{N}$-graded domain finitely generated over a field $R_0$. Then the valuation by grading extends to a $\mathbb{Q}$-valued valuation on the completion $\widehat{R_{R_+}}$ of $R$ which is non-negative on $\widehat{R_{R_+}}$ and positive on $\widehat{R_{R_+}}R_+$. Moreover, $\widehat{R_{R_+}}$ is again a domain.
\end{Le}
\begin{proof}
This is provided by the ``Untergrad''-function in \cite[proof of Satz 62.12]{schejastorch2} which also implies the last assertion. Alternatively, there is an ad-hoc argument in \cite[proof of Lemma 9.10]{staeblerthesis}.
\end{proof}

\begin{Prop}
\label{gradeddaggerdagger}
Let $R$ be an $\mathbb{N}$-graded domain finitely generated over a field $R_0$ of positive characteristic. Then graded dagger closure coincides with dagger closure as in Definition \ref{dagger}.
\end{Prop}
\begin{proof}
We prove that graded dagger closure coincides with tight closure in this setting. Then the result follows by Proposition \ref{daggertightexcellent}.

As before one inclusion is clear. Denote the completion of $R$ at $R_+$ by $S$ and note that it is an integral domain by the previous lemma. Now, if an element $f \in R$ is contained in $I^{\dagger \text{GR}}$ for some ideal $I \subseteq R$ then $f$ is multiplied into $IS$ by elements of arbitrarily small order with respect to the valuation $\nu$ in Lemma \ref{gradedvaluationextension}. Therefore, $f \in (IS)^\ast$ by \cite[Theorem 3.1]{hochsterhunekedagger}. Since $R$ has a completely stable test element (fields are excellent local rings -- thus \cite[Theorem 1.4.11]{hochsterhuneketightzero} applies) this implies $f \in I^\ast$.
\end{proof}

\begin{Ko}
\label{gradeddaggertight}
Let $R$ be an $\mathbb{N}$-graded ring finitely generated over a field $R_0$ of positive characteristic. Then graded dagger closure coincides with tight closure.
\end{Ko}
\begin{proof}
Since both definitions reduce to the domain case by killing minimal primes the result follows from the proof of Proposition \ref{gradeddaggerdagger}.
\end{proof}

\begin{Prop}
\label{daggerintegral}
Let $R$ be a noetherian domain. Then dagger closure coincides with integral closure for principal ideals and dagger closure is always contained in integral closure.
\end{Prop}
\begin{proof}
Let $I = (x)$ be a principal ideal in $R$ and let $S$ be the normalisation of $R$. If $f \in \bar{I}$ then $f \in \overline{IS} = IS$ (this follows from the first part of the proof of \cite[Proposition 10.2.3]{brunsherzog} which does not require the noetherian hypothesis) and hence immediately $f \in I^\dagger$.

For the other inclusion and the second assertion one can argue as in \cite[Proposition 2.6 (a)]{heitmannplusextended}).
\end{proof}

\section{Dagger closure in regular domains}
We now want to prove our main result, namely that for every ideal $I$ in a regular noetherian domain containing a field we have $I^\dagger = I$. The key ingredient will be the existence of big Cohen-Macaulay algebras.

\begin{Def}
Let $R$ be a noetherian local ring. A \emph{big balanced Cohen-Macaulay algebra} is an $R$-algebra $A$ such that the images of every system of parameters in $R$ form a regular system in $A$.
\end{Def}

The existence of such algebras has been proven by Hochster and Huneke if $R$ is an excellent local domain containing a field (see \cite{hochsterhunekeinfinitebig}). In fact, they also proved that if, in addition, $R$ is of positive characteristic then $R^+$ is a big balanced Cohen-Macaulay algebra. Moreover, Hochster and Huneke proved in \cite{hochsterhunekeapplicationsBCMA} that this assignment can be made weakly functorial in a certain sense.

\begin{Le}
\label{ValuationExtensionMapping}
Let $(R,\mathfrak{m})$ be a complete local noetherian domain containing a field and let $\nu$ be a rank one valuation on $R^+$ which is non-negative on $R$. Fix elements $r, x_1, \ldots, x_n$ in $R$. Then there is a big balanced Cohen-Macaulay algebra $A$ for $R$ such that if $r \in (x_1, \ldots, x_n)A$ then $\nu(r) \geq \min_i \nu(x_i)$.
\end{Le}
\begin{proof}
Note that $R$ is excellent.
By \cite[Theorem 5.6 (a)]{hochstertightsolid} (or \cite[Theorem 5.12]{hochsterhunekeapplicationsBCMA}), $r$ is then contained in the big equational tight closure of $(x_1, \ldots, x_n)$. Hence, it is a fortiori contained in the integral closure of $(x_1, \ldots, x_n)$. Now the result follows by the characterisation of integral closure in terms of valuations (cf.\ \cite[Proposition 10.2.4 (a)]{brunsherzog}).
\end{proof}

\begin{Le}
\label{ValuationExtensionFiniteNonnegative}
Let $(R, \mathfrak{m}) \subseteq (S, \mathfrak{n})$ be a finite extension of local domains. Let $\nu$ be a valuation on $S$ which is non-negative on $R$. Then $\nu$ is non-negative on $S$.
\end{Le}
\begin{proof}
Let $s \in S$, we find an equation $s^d = \sum_i r_i s^i$. Applying $\nu$ yields $d \nu(s) \geq \min_i \{\nu(r_i) + i \nu(s)\} \geq \min_i i \nu(s)$. Hence, $\nu(s) \geq 0$.
\end{proof}

\begin{Def}
Let $R$ be a domain and $M$ an $R$-module. We say that $m \in M$ is \emph{almost zero} if for every valuation $\nu$ on $R^+$ of rank at most one and for every $\eps > 0$ the element $m \otimes 1 \in M \otimes R^+$ is annihilated by an element $a \in R^{+}$ with $\nu(a) < \eps$.
\end{Def}

\begin{Theo}
\label{daggerregularclose}
Let $R$ be a regular noetherian domain containing a field $k$ and $I \subseteq R$ an ideal. Then $I^\dagger = I$. 
\end{Theo}
\begin{proof}
Localising at a maximal ideal $I \subseteq \mathfrak{m}$ and completing we may assume that $(R, \mathfrak{m})$ is a noetherian complete regular local domain by Lemma \ref{daggerextension} and since the completion of a regular noetherian local ring is again a regular noetherian local domain.

Since $I \subseteq J$ implies $I^\dagger \subseteq J^\dagger$ and since $I$ is the intersection of $\mathfrak{m}$-primary ideals (see e.\,g.\ \cite[Ex.\ 1.1]{hunekeapplication}) we may assume that $I$ is $\mathfrak{m}$-primary. 

Assume now that $I \neq I^{\dagger}$. Since $R/I$ has finite length we find by \cite[Theorem 3.1]{Eisenbud} an element $f \in I^\dagger/I \subset R/I$ whose annihilator is $\mathfrak{m}$. Similar to \cite[Proposition 2.8]{brennerstaeblerdaggersolid} an element of $R$ is in $I^{\dagger}$ if and only if its image is almost zero in $R/I$ (considered as an $R$-module).
Hence, $f$ is almost zero. Therefore, for every $\eps > 0$ there is an element $a \in R^+$ such that $f \otimes a = 0$ and $\nu(a) < \eps$. In particular, $af =0$ in $R^+/IR^+$ and this is the case if and only if $a f = \sum_i a_i f_i$ for suitable $a_i \in R^+$ and generators $f_1, \ldots, f_n$ of $I$. We may thus assume that all relevant data are contained in a finite ring extension $R \subseteq S$. Note that $S$ is still complete local by \cite[Corollary 7.6]{Eisenbud} since $S$ is a domain. 

By \cite[Theorem 3.9]{hochsterhunekeapplicationsBCMA} there is a weakly functorial big balanced Cohen-Macaulay algebra $A$ for $S$. We have an exact sequence $0 \to \mathfrak{m} \to R \to R \cdot f \to 0$ of $R$-modules. Since $A$ is flat over $R$ (by \cite[6.7]{hochsterhunekeinfinitebig} -- this is precisely where we need $R$ regular) tensoring with $A$ yields that the annihilator of $f$ in $A/IA$ considered as an $A$-module is $\mathfrak{m}A$. In particular, $a \in \mathfrak{m}A$.

In order to deduce a contradiction fix a rank one valuation $\nu$ on $R$ which is non-negative and positive on $\mathfrak{m}$ and extend it to $R^+$ (this is always possible and then necessarily a rank one valuation on $R^+$ -- see \cite[vol.\,2, VI, \S 4, Theorem 5']{zariskisamuel} and \cite[vol.\,2, VI, \S 11, Lemma 2]{zariskisamuel}). We consider the restriction of $\nu$ to $S$. By Lemma \ref{ValuationExtensionFiniteNonnegative} this valuation is non-negative on $S$. Hence, we can apply Lemma \ref{ValuationExtensionMapping} to see that $\nu(a) \geq \min_i \nu(x_i)$, where the $x_i$ are ideal generators of $\mathfrak{m}$. Since $\nu$ is actually positive on $\mathfrak{m}$ this is a contradiction for $\eps > 0$ small enough.
\end{proof}

\begin{Bem}
Since the assertion of the Theorem can be reduced to the complete local case the result follows in characteristic $p >0$ from \cite[Theorem 3.1]{hochsterhunekedagger} and from the fact that ideals in regular rings are tightly closed.
Furthermore, in positive characteristic, $R^+$ is a big balanced Cohen-Macaulay algebra for $R$  if $R$ is an excellent local domain (see \cite[Theorem 1.1]{hochsterhunekeinfinitebig} or \cite[Theorem 7.1]{hunekeapplication}). But this is wrong in characteristic zero if $\dim R \geq 3$.
\end{Bem}

\begin{Ko}
\label{regulargradeddagger}
Let $R = k[x_1, \ldots, x_d]$ be an $\mathbb{N}$-graded polynomial ring over a field $k = R_0$ and $I \subseteq R$ an ideal. Then $I^{\dagger \text{GR}} = I$.
\end{Ko}
\begin{proof}
One can employ essentially the same arguments as in the proof of Theorem \ref{daggerregularclose}. The reduction argument has only to be carried out with respect to a valuation as in Lemma \ref{gradedvaluationextension}.
\end{proof}

Note that since $R$ is $\mathbb{N}$-graded any regular ring with $R_0$ a field is in fact a polynomial ring -- see \cite[Exercise 2.2.25 (c)]{brunsherzog}.

\section{Consequences of the main theorem}
In this section we draw various consequences from our main theorem \ref{daggerregularclose}, discuss the one dimensional situation and compare dagger closure to Heitmann's full rank one closure.

\begin{Ko}
Let $R$ be the ring of invariants of a finite group $G$ acting linearly on a polynomial ring $A = k[x_1, \ldots, x_n]$ over a field $k$ of characteristic $p \geq 0$ such that $p$ does not divide the group order $\# G$. Then for any ideal $I \subseteq R$ we have $I^{\dagger \text{GR}} = I^\dagger = I$. 
\end{Ko}
\begin{proof}
We only prove the non-graded case.
If $f \in I^\dagger$ then we have $f \in (IA)^\dagger$ by Lemma \ref{daggerextension}. But by Theorem \ref{daggerregularclose} $(IA)^\dagger = IA$. Since $R$ is normal and $(p, \# G) = 1$ the field trace $\frac{1}{d} \Tr$ splits the inclusion $R \to A$, where $d$ is the degree of $Q(A)$ over $Q(R)$. In particular, we have $IA \cap R = I$.
\end{proof}

Recall that for a local noetherian ring $(R, \mathfrak{m})$ of dimension $d$ and parameters $(x_1, \ldots, x_d)$ an element $1/(x_1 \cdots x_d)$ in $H^d_{\mathfrak{m}}(R)$ is called a \emph{paraclass}.

\begin{Ko}
\label{Paraclassnotalmostzero}
Let $(R,\mathfrak{m})$ be an excellent local domain of dimension $d$ containing a field $k$. Then no paraclass $c \in H^d_\mathfrak{m}(R)$ is almost zero.
\end{Ko}
\begin{proof}
Let $x_1, \ldots, x_d$ be parameters and $c = 1/(x_1 \cdots x_d)$ the corresponding para\-class.
We may assume that $R$ is normal. For if $c$ is almost zero then a fortiori $c$ is almost zero in $H^d_{\mathfrak{m}S}(S)$, where $S$ is the normalisation of $R$.
Since $H^d_\mathfrak{m}(R) = H^d_\mathfrak{m}(\widehat{R})$ and $R$ is excellent and normal we may moreover assume that $R$ is a complete local domain.

By virtue of complete Noether normalisation (see \cite[Theorem A.22]{brunsherzog}) the ring $T = k[[x_1, \ldots, x_d]]$ is a regular local ring over which $R$ is finite. In particular, $T^+ = R^+$ and since the annihilator of $c \in H^d_\mathfrak{m}(T)$ is $(x_1, \ldots, x_d)$ (cf. \cite[Lemma 2.8]{brennerparasolid}) it cannot be almost zero by arguments as in the proof of Theorem \ref{daggerregularclose}.
\end{proof}

Next we shall need the notion of a parasolid algebra as introduced in \cite{brennerparasolid}.
Since we are not concerned with the mixed characteristic case here we may state a simpler definition.

\begin{Def}
Let $R$ denote a local noetherian ring of dimension $d$ containing a field. An $R$-algebra $A$ is called \emph{parasolid} if the image of every paraclass $c \in H^d_\mathfrak{m}(R)$ in $H^d_\mathfrak{m}(A)$ does not vanish.

An algebra $A$ over noetherian ring $R$ containing a field is called \emph{parasolid} if $A_\mathfrak{m}$ is parasolid over $R_\mathfrak{m}$ for every maximal ideal $\mathfrak{m}$ of $R$.
\end{Def}

\begin{Le}
\label{parasolidfinite}
Let $R \subseteq S$ be a finite extension of noetherian domains containing a field $k$ and let $A$ be an $R$-algebra. If $A \otimes_R S$ is parasolid as an $S$-algebra then $A$ is parasolid.
\end{Le}
\begin{proof}
We may assume that $R$ is a local complete noetherian ring with maximal ideal $\mathfrak{m}$ and that $R \subseteq S$ is quasi-local (cf.\ \cite[Proposition 1.5]{brennerparasolid}, \cite[Corollary 7.6]{Eisenbud} and recall that completion is flat). Fix a paraclass $c = 1/(x_1 \cdots x_d)$ in $H^d_\mathfrak{m}(R)$. Let $\mathfrak{n}$ be a maximal ideal in $S$ containing $(x_1, \ldots, x_d)$. Localising $S$ at $\mathfrak{n}$ one has that $x_1, \ldots, x_d$ are parameters for $S_\mathfrak{n}$. In particular, the image of $c$ in $H^d_{\mathfrak{n}S_\mathfrak{n}}(S_\mathfrak{n})$ is nonzero.
We have a commutative diagram
\[
\begin{xy}
 \xymatrix{H^d_{\mathfrak{n}S_\mathfrak{n}}(S_\mathfrak{n}) \ar[r]& H^d_{\mathfrak{n}S_\mathfrak{n}}(A \otimes_R S_\mathfrak{n})\\
H^d_\mathfrak{m}(R) \ar[u] \ar[r] & H^d_\mathfrak{m}(A) \ar[u]}
\end{xy}
\]
and since $A \otimes_R S$ is parasolid $c$ cannot vanish in $H^d_\mathfrak{m}(A)$.
\end{proof}

\begin{Ko}
\label{daggerimpliesparasolid}
Let $R$ be an excellent domain of dimension $d$ containing a field and $I = (f_1, \ldots, f_n)$ an ideal. If $f \in I^\dagger$ then the forcing algebra for $(f, I)$ is parasolid.
\end{Ko}
\begin{proof}
Since we may pass to finite ring extensions we can assume that $R$ is normal (this is clear for dagger closure and follows from Lemma \ref{parasolidfinite} for the parasolid case).
By Lemma \ref{daggerextension} we have that $f \in (IR_\mathfrak{m})^\dagger$ for every maximal ideal $\mathfrak{m}$ of $R$ and we may also pass to the completion (this does not affect parasolidity -- see \cite[Proposition 1.5]{brennerparasolid}). Hence, we may assume that $(R,\mathfrak{m})$ is a complete local domain of dimension $d$.

Assume to the contrary that the forcing algebra \[ B = R[T_1, \ldots, T_n]/(\sum_{i=1}^n f_i T_i - f_0) \] is not parasolid. That is, there exists a paraclass $1/(x_1 \cdots x_d)$ coming from $H^d_\mathfrak{m}(R)$ which vanishes in $H^d_\mathfrak{m}(B)$. This is the case if and only if we have an equation \[b_1 x_1^{t+1} + \ldots + b_d x_d^{t+1} = (x_1 \cdots x_d)^t,\] where $b_i \in B$ (see \cite[Remark 9.2.4 (b) and the discussion at the beginning of Section 9.3]{brunsherzog}).

Since $f \in I^\dagger$ we have for every $\eps >0$ a relation $a_0 f_0 = \sum_{i=1}^n a_i f_i$ with $a_i \in R^+$ and $\nu(a_0) < \eps$. Hence, in $Q(R^+)$ we have $f_0 = \sum_{i=1}^n \frac{a_i}{a_0} f_i$. These relations induce homomorphisms \[\varphi_\eps: B \longrightarrow Q(R^+), \quad T_i \longmapsto \frac{a_i}{a_0}.\] This in turn induces a relation \[a_0^l \sum_{j=1}^d \varphi_\eps(b_j) x_j^{t+1} = a_0^l (x_1 \cdots x_d)^t\] in the $x_j$ with coefficients in $R^+$ for sufficiently large $l$. Since the $b_j$ are polynomials in the $T_i$ the exponent $l$ does not depend on $a_0$. Thus we find that $1/(x_1 \cdots x_d) \in H^d_\mathfrak{m}(R)$ is almost zero. This is a contradiction in light of Corollary \ref{Paraclassnotalmostzero}.
\end{proof}

\begin{Bem}
Since we do not have general persistence results for dagger closure we cannot extend Corollary \ref{daggerimpliesparasolid} to the universally parasolid case.
The result of Corollary \ref{daggerimpliesparasolid} also holds for graded dagger closure if we only consider homogeneous ideals and only localise at $R_+$ by virtue of Lemma \ref{gradedvaluationextension}.
\end{Bem}

This result allows us to recover (and extend) an inclusion we proved in \cite{brennerstaeblerdaggersolid} using geometric methods. We refer to \cite{hochstersolid} for a definition and background on solid closure.

\begin{Ko}
\label{gradeddaggerinsolidgeneral}
Let $R$ be an $\mathbb{N}$-graded domain finitely generated over a field $k = R_0$ of dimension $d$ and $I$ a homogeneous ideal of height $d$ (e.\,g.\ $R_+$-primary). Then graded dagger closure is contained in solid closure.
\end{Ko}
\begin{proof}
To begin with, we may assume $R$ to be normal. Then for a given ideal $I$ and an element $f$ in $R$ with forcing algebra $A$ we have that $f \in I^\star$ if and only if $H^d_{R_+}(A) \neq 0$ (note that since $R$ is normal and excellent $R' = \widehat{R_{R_+}}$ is integral, and $H^d_{\widehat{R_+}}(R' \otimes_R A) = H^d_{R_+}(A)$ by flat base change \cite[Theorem 4.3.2]{brodmannsharp}, since $H^d_\mathfrak{m}(A) = H^d_\mathfrak{m}(R) \otimes A$ and because $H^d_{\widehat{\mathfrak{m}}}(\widehat{R}) = H^d_\mathfrak{m}(R)$. Moreover, $R_+$ is the only maximal ideal over $I$ we only need to look at local cohomology with respect to $R_+$). But by the previous remark $f \in I^{\dagger \text{GR}}$ implies $H^d_{R_+}(A) \neq 0$.
\end{proof}

\begin{Bem}
We remark that the condition that $\height I$ be equal to the dimension of $R$ is necessary for this approach. For if $\height I < \dim R$ then there is some non-homogeneous maximal ideal $\mathfrak{m}$ containing $I$. The valuation on $R$ extends to $R_\mathfrak{m}$ but it is no longer non-negative. Hence, any element of $H^d_\mathfrak{m}(A)$ is almost zero with respect to this valuation.
\end{Bem}

This Corollary also has interesting geometric consequences -- cf.\ \cite[Remark 2.8 (b)]{staeblerdaggerabelian}.

\begin{Ko}
\label{daggerinsolid}
Let $R$ be an excellent domain of dimension $d$ containing a field. Then dagger closure is contained in solid closure. 
\end{Ko}
\begin{proof}
We may assume that $R$ is normal.
Let $I$ be an ideal and $f \in I^\dagger$. Then by Lemma \ref{daggerextension} we have that $f \in (IR_\mathfrak{m})^\dagger$ for every maximal ideal $\mathfrak{m}$ containing $I$. Thus Corollary \ref{daggerimpliesparasolid} implies that $H^d_\mathfrak{m}(A) \neq 0$, where $A$ is a forcing algebra for the data $R_\mathfrak{m}, f, I$. Hence, $f \in I^\star$.
\end{proof}

\begin{Bem}
Note that both the inclusions of Corollary \ref{gradeddaggerinsolidgeneral} and \ref{daggerinsolid} are strict in general. This is due to Roberts' example showing that ideals in a regular ring of dimension $\geq 3$ need not be solidly closed (cf.\ \cite[Discussions 7.22 and 7.23 or Corollary 7.24]{hochstersolid}). Specifically, one has $x^2y^2z^2 \in (x^3,y^3,z^3)^\star$ in $\mathbb{Q}[x,y,z]$ but $x^2y^2z^2 \notin (x^3,y^3,z^3)^\dagger$ by Theorem \ref{daggerregularclose}.
\end{Bem}

\begin{Ko}
\label{gradedaggerinintegral}
Let $R$ be an $\mathbb{N}$-graded domain finitely generated over a field $k = R_0$ and let $I$ be a homogeneous ideal of $\height I = \dim R$. Then $I^{\dagger \text{GR}} \subseteq \bar{I}$.
\end{Ko}
\begin{proof}
Solid closure is contained in integral closure for noetherian rings by \cite[Theorem 5.10]{hochstersolid}. Thus the result follows from Corollary \ref{gradeddaggerinsolidgeneral}.
\end{proof}

As another immediate application we obtain an exclusion bound for graded dagger closure. This was proven for tight closure in \cite[Theorem 2.2]{smithgraded}.

\begin{Ko}
Let $R$ be an $\mathbb{N}$-graded domain of dimension $d \geq 2$ finitely generated over an algebraically closed field $k = R_0$ and $(f_1, \ldots, f_n)$ an $R_+$-primary homogeneous ideal, where the $f_i$ are homogeneous of degrees $d_i$. Let $f_0$ be another homogeneous element of degree $d_0 \leq \min_i d_i$. Then $f_0 \in (f_1, \ldots, f_n)^{\dagger \text{GR}}$ is only possible if $f_0 \in (f_1, \ldots, f_n)$. Likewise, $f_0 \in (f_1, \ldots, f_n)^{\dagger}$ is in this case only possible if $f_0 \in (f_1, \ldots, f_n)$.
\end{Ko}
\begin{proof}
Passing to a finite ring extension we may assume that $R$ is normal and that $R(1)^\sim$ on $\Proj R$ is globally generated (hence invertible).
If $f_0$ is contained in the graded dagger closure of $(f_1, \ldots, f_n)$ then we must have $f_0 \in (f_1, \ldots, f_n)^\star$ by Corollary \ref{gradeddaggerinsolidgeneral}. But this is only possible if $f_0 \in (f_1, \ldots, f_n)$ by virtue of \cite[Corollary 4.5]{brennertightproj} (here we need the conditions on $R(1)^\sim$). The same argument works for dagger closure if we use Corollary \ref{daggerinsolid} instead of Corollary \ref{gradeddaggerinsolidgeneral}.
\end{proof}

\noindent We also note that this result is trivial if $d_0 < \min_i d_i$.

We now compare (graded) dagger closure to integral closure in one dimensional rings.

\begin{Le}
\label{integraldim1}
Let $R$ be a one-dimensional regular noetherian domain and $I$ an ideal in $R$. Then $\bar{I} = I$.
\end{Le}
\begin{proof}
If $R$ is local then the assertion follows from \cite[Proposition 10.2.3]{brunsherzog}. In general, we have $\bar{I} \subseteq \bigcap_P \bar{I}R_P = \bigcap_P IR_P = I$, where the intersection runs over all primes $P \in \Spec R$ and the latter equality holds since $I^\sim$ is a sheaf on $\Spec R$.
\end{proof}

\begin{Ko}
\label{gradeddaggerintegralcl}
Let $R$ be an $\mathbb{N}$-graded ring of dimension one which is finitely generated over a field $R_0$ and let $I \subseteq R$ an ideal. Then $\bar{I} = I^{\dagger \text{GR}}$.
\end{Ko}
\begin{proof}
We may reduce to the domain case. This immediately follows from the definition for graded dagger closure and from \cite[Proposition 10.2.2 (c)]{brunsherzog} for integral closure.
Denote the normalisation of $R$ by $S$. We then have $I^{\dagger \text{GR}} \subseteq (IS)^{\dagger \text{GR}} \cap R = IS \cap R$ by Corollary \ref{regulargradeddagger}. Since the inclusion $IS \cap R \subseteq I^{\dagger \text{GR}}$ is immediate equality holds.

Furthermore, we have $\bar{I} \subseteq \overline{IS} \cap R = IS \cap R$ by Lemma \ref{integraldim1}. If $f \in IS \cap R$ then write $f = \sum_i s_it_i$, where $t_i \in I$ and $s_i \in S$. As $\bar{I}$ is an ideal we only need to show that $s_i t_i \in \bar{I}$ for each $i$. So fix $i$ and omit the index. Since $R \subseteq S$ is integral we have an equation $s^m + a_1 s^{m-1} + \ldots + a_{m-1}s + a_m = 0$ with $a_i \in R$. Therefore, $(st)^m + a_1t (st)^{m-1} + \ldots + a_{m-1}t^{m-1} (st) + a_m t^m = 0$ yields that $f \in \bar{I}$.
\end{proof}

\begin{Ko}
Let $R$ be a one-dimensional noetherian domain containing a field $k$ such that the normalisation is finite over $R$ (e.\,g.\ $R$ excellent or, more generally, Japanese\footnote{See \cite[Chapitre 0, D\'efinition 23.1.1]{EGAIV} for a definition of Japanese rings.}). Then $I^\dagger = \bar{I}$ for any ideal $I$.
\end{Ko}
\begin{proof}
See the proof of Corollary \ref{gradeddaggerintegralcl}. Also note that for principal ideals this is already stated in Proposition \ref{daggerintegral}.
\end{proof}

In the remainder of this article we briefly address the relation of our definition of dagger closure with Heitmann's full rank one closure (see \cite[Definition after Remark 1.3]{heitmannplusextended}) which we shall recall here for the convenience of the reader: 

\begin{Def}
\label{DefConvenience}
Let $R$ be a domain and $I$ an ideal in $R$. Then the \emph{full rank one closure} of $I$ is given by the set of elements $f \in R$ such that for every valuation $\nu$ on $R^+$ of rank at most one, every prime number $p$, every positive integer $n$ and every $\eps > 0$ there exists $a \in R^+$ with $\nu(a) < \eps$ such that $af \in (I, p^n)R^+$.
\end{Def}

Note that Definition \ref{dagger} and Heitmann's definition coincide in positive characteristic. In particular, we have

\begin{Ko}
Let $R$ be a domain of characteristic $p > 0$ essentially of finite type over an excellent local ring. Then Heitmann's full rank one closure coincides with tight closure.
\end{Ko}
\begin{proof}
This is immediate from Proposition \ref{daggertightexcellent}.
\end{proof}

 Obviously, full rank one closure is not applicable in equal characteristic zero. In this case Heitmann proposed a different definition which reduces to the mixed characteristic case by looking at a finitely generated $\mathbb{Z}$-algebra containing all relevant data (see \cite[Definition after Lemma 1.6]{heitmannplusextended}).

As mentioned before we do not have any results in mixed characteristic. Nevertheless, we propose the following alternative definition in the local case which is the same as dagger closure if the ring contains a field and coincides with Heitmann's rank one closure in mixed characteristic.

\begin{Def}
Let $(R,\mathfrak{m},k)$ be a local domain, $I$ an ideal and $\chara k = p \geq 0$. Then we define the \emph{local dagger closure} $I^\dagger$ of $I$ as the set of elements $f \in R$ such that for every $\eps > 0$, every $n \in \mathbb{N}$ and every valuation $\nu$ of rank at most one on $R^+$  there exists $a \in R^+$ with $\nu(a) < \eps$ and $af \in (I, p^n)R^+$.
\end{Def}

Of course, if  $R$ is a local domain of mixed characteristic we have that any prime $q \neq \chara k$ is invertible in $R$. Hence, we only need to look at $p = \chara k$ in the full rank one closure so that the two definitions agree.

It should also be possible to extend these definitions (and most of the results) to rings which are not domains by reduction to the domain case (i.\,e.\ via killing minimal primes) but we did not pursue this (see also \cite{heitmannplusextended}). We also made no systematic effort of extending Heitmann's results concerning persistence to dagger closure in equal characteristic zero.

\section*{Acknowledgements}
We would like to thank the referee for a careful reading of an earlier draft of this article and for pointing out a major simplification in the proof of our main theorem.

\bibliography{bibliothek.bib}
\bibliographystyle{amsplain}
\end{document}